\documentclass[12pt]{amsart}

\usepackage{amsthm,amssymb,amstext,amscd,amsfonts,amsbsy,amsxtra,latexsym,amsmath}
\usepackage{fullpage}
\usepackage[english]{babel}
\usepackage[latin1]{inputenc}
\usepackage{verbatim}
\usepackage{graphicx} 
\usepackage[all]{xy} 
\numberwithin{figure}{section}

\newcommand{\field}[1]{\mathbb{#1}}
\newcommand{\N}{\field{N}}
\newcommand{\Z}{\field{Z}}
\newcommand{\R}{\field{R}}

\renewcommand{\H}{\mathbb{H}}

\newtheorem{theorem}{\textbf{Theorem}}[section]
\newtheorem{corollary}[theorem]{\textbf{Corollary}}
\newtheorem{proposition}[theorem]{\textbf{Proposition}}
\newtheorem{definition}[theorem]{\textbf{Definition}}
\newtheorem*{example}{Example}
\newtheorem*{remark}{Remark}

\renewenvironment{proof}[1][Proof]{\begin{trivlist}
\item[\hskip \labelsep {\bfseries #1:}]}{\qed\end{trivlist}}

\newcommand{\bea}{\begin{eqnarray}} 
\newcommand{\eea}{\end{eqnarray}} 
\newcommand{\be}{\begin{equation}} 
\newcommand{\ee}{\end{equation}} 
\newcommand{\benn}{\begin{equation*}} 
\newcommand{\eenn}{\end{equation*}} 
\thanks{The author's research is supported by the DFG-Graduiertenkolleg 1269 ''Globale Strukturen in Geometrie
und Analysis'' \\ This paper is part of the author's PhD-thesis, written under the supervision of Prof. Dr. K.
Bringmann at the university of Cologne.}

\title[]{Asymptotics of higher Order Ospt-functions for Overpartitions}
\author{Jose Miguel Zapata Rolon}
\address{ Mathematical Institute\\University of
Cologne\\ Weyertal 86-90 \\ 50931 Cologne \\Germany}
\email{rzapata@math.uni-koeln.de}

\begin{document}

\begin{abstract}
In this paper we obtain asymptotic formulas for the positive crank and rank moments for overpartitions. Moreover, we show that crank and rank moments are asymptotically equal while the difference is asymptotically positive. This indicates that there exist analogous higher ospt-functions for overpartitions, which we define. \end{abstract}

\maketitle
\section{Introduction}
A partition $\lambda$ of $n$ is a non-increasing sequence of positive integers whose sum is $n$. The theory of partitions gives many impressive examples of connections between automorphic forms and combinatorics. Ramanujan proved the following congruences  for $p\left(n\right)$, the counting function for the number of partitions of $n$ \cite{Ram2},
\begin{eqnarray}\label{racon}
 p(5n+4)&\equiv & 0 \pmod{5}, \nonumber\\
 p(7n+5)&\equiv & 0  \pmod{7} , \\
 p(11n+6)&\equiv& 0  \pmod{11} \nonumber.
\end{eqnarray}
These congruences are precursors for the general theory of Hecke operators on $l$-adic modular forms. Ramanujan's proof gave little insight into the combinatorial nature of these congruences. Therefore Dyson invented the concept of the rank and conjectured the existence of the crank \cite{Dy1} to explain the congruences from a combinatorial point of view. The crank was finally constructed by Garvan and Andrews \cite{AG1}. The rank is defined
in the following way:
\begin{equation}
\textrm{rank}(\lambda) := \textrm{largest part of}\, \lambda - \textrm{number of parts of} \, \lambda ,
\end{equation}
and the crank as: 
\begin{equation}
\textrm{crank}(\lambda):=\begin{cases}  \textrm{largest part of}\, \lambda \qquad &\textrm{if}\, o(\lambda) = 0, \\
					\mu(\lambda) - o(\lambda)	   \qquad &\textrm{if}\, o(\lambda) > 0, 	
                         \end{cases}
\end{equation}
where  $o(\lambda)$ is the number of ones in a partition $\lambda$ and $\mu(\lambda)$ is the number of parts larger than $o(\lambda)$.
The rank explains the first two congruences in (\ref{racon}) \cite{AS-D}
and the crank all three simultaneously \cite{AG1} by grouping together the different partitions into equally sized congruence classes of rank or crank values. There are many other congruences for the partition function, however the author showed that theses congruences (\ref{racon}) are the only ones that can be explained by the crank \cite{Za}. 
The study of the generating functions of these combinatorial objects led to many new examples of automorphic forms like harmonic weak Maass forms (see e.g. \cite{BO}). Another way to build out of the partition statistics examples of new functions is to use rank and crank moments.
\begin{definition}
Let $M(m,n) \,(resp. \, N(m,n))$ be the number of partitions of $n$ with crank  (resp. rank) $m$ . Then for $r \in \N_0$ we define the $r$-th crank (resp. rank)-moment $M_r \,(resp. \, N_r)$ as
\begin{align*}
M_{r}\left(n\right) := &\sum_{m \in \Z} m^r M(m,n), \\
N_{r}\left(n\right) := &\sum_{m \in \Z} m^r N(m,n).
\end{align*}
\end{definition} 
It can be easily seen that the odd moments vanish due to the symmetry $M(-m,n)= M(m,n)\,(\text{resp.}\, N(-m,n)= N(m,n))$. %Therefore, to study non-trivial odd moments we consider sums over just the natural numbers.  
%\begin{definition}
% For $r\in \N_0$, we define the $r$-th crank(resp. rank)-moment $M_r^{+}\, (\text{resp.} \, N_r^{+})$ as
%\begin{align*}
%M_{r}^{+}\left(n\right) := &\sum_{m \in \N} m^r M(m,n), \\
%N_{r}^{+}\left(n\right) := &\sum_{m \in \N} m^r N(m,n).
%\end{align*}
%\end{definition} 

A natural generalization of the theory of partitions is the theory of overpartitions \cite{CL}. An overpartition of $n$ is a non-increasing sequence of positive integers whose sum is $n$ with the extra condition that you may overline the first occurrence of any number.
For example, the 8 overpartitions of 3 are:
\begin{align*}
   3,\, \overline{3},\, 2 + 1,\, \overline{2}+1,\, 2+\overline{1},\, \overline{2}+\overline{1},\, 1+1+1,1+1+ \overline{1}. 
\end{align*}
It is also possible to define ranks and cranks of overpartitions. There are several different possibilities. We restrict in our research to Dyson's rank defined in the same way as for partitions. The first residual crank for overpartitions is obtained by taking the crank of the subpartition consisting of the non-overlined parts. 
We denote the number of overpartitions of $n$ with rank $m$ by $\overline{N}(m,n)$  and the number of overpartitions of $n$ with crank $m$ by $\overline{M}(m,n)$. These functions fulfill again the equations $\overline{M}(-m,n)= \overline{M}(m,n)$ and $\overline{N}(-m,n)= \overline{N}(m,n)$.
Therefore, to study non-trivial odd moments we consider positive moments as sums over the natural numbers.
\begin{definition} Suppose that $r \in \N$. The $r$-th positive crank (resp. rank) moment is defined by
\begin{align*}
\overline{M}_{r}^{+}\left(N\right) := & \sum_{m \in \N} m^r \overline{M}\left(m,N\right), \\
\overline{N}_{r}^{+}\left(N\right):= & \sum_{m \in \N}  m^r \overline{N}\left(m,N\right).
\end{align*}
\end{definition}
The technical problem that arises with positive moments is that a certain symmetry of $\Z$ can not be used. This will lead to the introduction of false theta series that make computations much harder, because this class of functions is not modular.
Modularity is an important tool to compute Fourier coefficients, however we will overcome this problem by using Wright's Circle Method \cite{W}. 

\begin{remark} 
If the index is even we have the following equation
\begin{align*}
\overline{M}_{2r}^{+}\left(N\right)&= \frac{1}{2}\overline{M}_{2r}\left(N\right), \\
\overline{N}_{2r}^{+}\left(N\right)&= \frac{1}{2}\overline{N}_{2r}\left(N\right).
\end{align*} 
So for even index we get the usual $k$-moments.
\end{remark}

In this paper we prove the following statement
\begin{theorem}\label{th1}
 Let $ r \in \N $. Then as $N \rightarrow \infty$, we have
\begin{itemize}
 \item[(i)]
\benn
\overline{M}_r^+(N) \sim \overline{N}_r^+(N) \sim \gamma_r N^{\frac{r}{2}-1} e^{\pi\sqrt{N}},
\eenn
where
\benn
\gamma_r= r!\zeta(r)(1-2^{1-r})\pi^{-r}2^{r-3}.
\eenn
Here $\zeta(r)$ denotes the Riemann $\zeta$-function. 
\item[(ii)]
As $N \rightarrow \infty$, we have
\benn
\overline{M}_r^+(N) - \overline{N}_r^+(N) \sim \delta_r N^{\frac{r}{2}-\frac{3}{2}} e^{\pi\sqrt{N}},
\eenn
where
\benn
\delta_r= r!\pi^{-r+1}2^{r-\frac{7}{2}}\left(\zeta(r-2)(1-2^{3-r})+ \frac{1}{2}\zeta(r-1)(1-2^{2-r})\right).
\eenn
\end{itemize}
\end{theorem}
\begin{remark}
 Note that we will use the analytic continuation of the Riemann $\zeta$-function for small values of $r$. That is the reason why the proof for small $r$ differ.
\end{remark}
To state the corollary we introduce the following quantity
\begin{definition}
Let $ r \geq 1 $. Then we define the $r$-th $ospt$-function for overpartitions as
%\begin{align*}
%\overline{spt}_{r}^+(N)&:= \left(\overline{\mu}_{2r}^+(N) - \overline{\eta}_{2r}^+(N)\right) \\
% \overline{spt}_{r,odd}^+(N)&:= \left(\overline{\mu}_{2r-1}^+(N) - \overline{\eta}_{2r-1}^+(N)\right).
%\end{align*}
$$
\overline{ospt}_r\left(N\right):= \overline{M}_r^+\left(N\right) - \overline{N}_r^+\left(N\right).
$$
\end{definition}
From the previous theorem, we obtain directly the following result:
\begin{corollary}\label{co1}
 For all $r \in \N$ there exists a $n_r \in \N$ such that if $N \geq n_r $, then
\benn
\overline{M}_r^+(N) - \overline{N}_r^+(N) > 0,
\eenn
so moreover for $r \geq 1$ and $N \geq n_r$ we have
$$
\overline{ospt}_{r}(N)> 0.
$$
\end{corollary}

\begin{remark}
Computer experiments show that the inequality $\overline{M}^+_r(N) - \overline{N}^+_r(N)>0$ hold for every $N,r \geq 1$. Therefore it would be worthwhile to investigate these new-defined higher order $ospt$-functions and to show what they actually count. %Due to the non-modularity of the generating functions, explicit bounds combined with computer calculations do not seem to be an easy way to show the statement in general. 
A recent paper shows that the inequality hold for general $N$ if $r$ is even \cite{J-S}.
\end{remark}

\begin{remark}
 In \cite{ABKO} it is shown that $\overline{M}_1^+(N) - \overline{N}_1^+(N) > 0$ for all $N$. The authors obtain their result by explicitly computing the generating functions of the first rank and crank moment. 
\end{remark}

%\begin{remark}
%This is only an asymptotic statement and it would be very interesting to find a combinatorial definition of the $spt$ functions and to show so that it holds for all $N$ and $r$. 
%\end{remark}
The remainder of the paper is organized as follows: In Section 2 we will introduce the symmetrized crank and rank moments and derive their generating functions. Next, in Section 3, we compute the asymptotic behavior of the generating functions using Mittag-Leffler decomposition and Taylor expansions.  In Section 4 we use Wright's Circle Method
to obtain asymptotics of the Fourier coefficients of the symmetrized moment generating function, completing the proof of Theorem \ref{th1}. 
\section{Constructing generating functions}
The generating functions of the symmetrized moments are much more convenient to work with than with the combinatorial definition made in the introduction. The positive symmetrized moments are defined in the following way.
\begin{align*}
\overline{\mu}_{r}^{+}\left(n\right) := & \sum_{m \in \N}^{} \binom{m + \lfloor \frac{r-1}{2} \rfloor}{r}\overline{M}(m,n), \nonumber \\
\overline{\eta}_{r}^{+}\left(n\right):= & \sum_{m \in \N} \binom{m + \lfloor \frac{r-1}{2} \rfloor}{r}\overline{N}(m,n). \nonumber
\end{align*}
We prove our main result by computing the asymptotics of the symmetrized moments which are related to the positive $r$-moments by the following equation \cite{G}:
\bea \label{smum}
\overline{M}_{r}^{+}(n) = r! \overline{\mu}_{r}^{+}(n)  + \sum_{l=0}^{r-1} a_l\overline{\mu}_{l}^{+}(n), \\ 
\overline{N}_{r}^{+}(n) = r! \overline{\eta}_{r}^{+}(n) + \sum_{l=0}^{r-1} a_l\overline{\eta}_{l}^{+}(n), \nonumber 
\eea
where $a_l$ are numbers and due to this formula will have the same asymptotics. 
The generating functions of the different moments are calculated by using differential operators for the corresponding generating functions. 
These generating functions are given by the following expressions (see \cite{BLO, L}):
\begin{eqnarray*}
 \overline{C}(z;q) := \sum_{m \in \Z}\sum_{n \in \N} \overline{M}(m,n) z^m q^n = \frac{(-q)_{\infty}(q)_{\infty}}{(zq)_{\infty}(z^{-1}q)_{\infty}} \\
 \overline{R}(z;q) := \sum_{m \in \Z}\sum_{n \in \N} \overline{N}(m,n) z^m q^n = \sum_{n = 0}^{\infty} \frac{(-1)_{n}q^{\frac{n(n+1)}{2}}}{(zq)_n (z^{-1}q)_n} 
\end{eqnarray*}
We introduce the operator $P$ that projects the function that has negative exponents in the $z$-expansion to the same expansion but keep just terms where the exponent of $z$ is positive. Furthermore, we introduce the evaluation operator $\varepsilon_{\alpha}: f( \bullet, q) \rightarrow f(\alpha,q) $.
For the usual positive $k$-moments we use the following operator:
\benn
\mathcal{M}_k := \varepsilon_{1} \circ P  \circ \left(z\dfrac{\partial}{\partial z}\right)^k .
\eenn
Finally, for the symmetrized moments (which will be the interesting function throughout this paper) we use:
\be \label{op}
\mathcal{SM}_k :=\varepsilon_{1} \circ P \circ \frac{1}{r!}\dfrac{\partial^{r}}{\partial z^{r} } z^{m + \lfloor \frac{r-1}{2}\rfloor} .
\ee
\begin{remark}
We do not address the use of the usual moment operator here.
\end{remark}
To give uniform formulas for general $r$, we define 
\begin{align*}
 \rho_C(r):= \rho_C := \begin{cases}
           0     \,\,\, \textrm{if r is odd}, \\
           \frac{1}{2}\,\,\, \textrm{otherwise},
          \end{cases}
\end{align*}
and
\begin{align*}
\rho_R(r):= \rho_R := \begin{cases}
           \frac{1}{2}   \,\,\, \textrm{if r is odd} ,\\
           1 \,\,\, \textrm{otherwise}.
          \end{cases}
\end{align*}
We obtain the following result:
\begin{proposition}
The generating functions of the symmetrized moments are
\begin{eqnarray*}
\mathcal{SC}_r(q):= \sum_{n \geq 0} \overline{\mu}_{r}^{+}(n)q^n  = \mathcal{SM}_r \overline{C}(z;q) =  \frac{(-q)_{\infty}}{(q)_{\infty}}\sum_{n\geq 1} \frac{(-1)^{n+1}q^{\frac{n^2}{2}+\left(\frac{r}{2}+\rho_{C}\right)n}}{\left(1-q^n\right)^r},\\
\mathcal{SR}_r(q):= \sum_{n \geq 0} \overline{\eta}_{r}^{+}(n)q^n = \mathcal{SM}_r \overline{R}(z;q) = 2\frac{(-q)_{\infty}}{(q)_{\infty}}\sum_{n\geq 1} \frac{(-1)^{n+1}q^{n^2 +\left(\frac{r}{2}+\rho_{R}\right)n}}{\left(1+q^n\right)\left(1-q^n\right)^r}.\\
\end{eqnarray*}
\end{proposition}
 It is a straightforward computation to use the differential operator on the corresponding generating function, so we skip the proof. 
\begin{example}
 Here we give two examples of $q$-series that are generating functions of certain positive symmetrized rank/crank moments. 
\benn
\mathcal{SR}_3(q)= \frac{(-q)_{\infty}}{(q)_{\infty}}\sum_{n\geq 1} \frac{(-1)^{n+1}q^{\frac{n^2 + 3n}{2}}}{\left(1-q^n\right)^3} =  2 q^3 + 8 q^4 + 24 q^5 + 60 q^6 + 134 q^7 + ...
\eenn
and 
\benn
\mathcal{SC}_4(q)=  2\frac{(-q)_{\infty}}{(q)_{\infty}}\sum_{n\geq 1} \frac{(-1)^{n+1}q^{n^2 + 2n}}{\left(1+q^n\right)\left(1-q^n\right)^4} = q^2 + 6 q^3 + 22 q^4 + 63 q^5  + 159 q^6  + 358 q^7 + ...
\eenn
\end{example}
\section{Asymptotic expansion of the generating functions}
In this section, we investigate the asymptotic behavior of the generating functions of the symmetrized moments near an essential singularity on the unit circle.
To do so, we fix the notation as follows. Let $q := e^{2\pi i \tau}$, where $\tau = x + iy$, with $x,y \in \R$ and $y>0$. The generating functions $\mathcal{SC}_r(q), \mathcal{SR}_r(q)$ have exponential singularities coming from the factor
$ (-q)_{\infty}/(q)_{\infty}$ that mainly control the asymptotic behavior. The dominant pole is at $q=1$, so we have to deduce the asymptotic behavior of the rest of the function near this point.
In general, we will show that it has polynomial growth. 
\subsection{Bounds near the dominant pole}
We prove the following result strongly related to the asymptotics obtained in \cite{BM}:
\begin{proposition}
Let 
$$ S_r(q) := \sum_{n\geq1}\frac{(-1)^{n+1}q^{\frac{n^2}{2} + \left(\frac{r}{2}+\rho_{C}\right)n}}{\left(1-q^n\right)^r}, $$
and 
$$ \widetilde{S}_r(q):= \sum_{n\geq1}\frac{(-1)^{n+1}q^{n^2 + \left(\frac{r}{2}+\rho_{R}\right)n}}{\left(1-q^n\right)^r \left(1+q^n\right)}. $$
Letting $|x|\leq y := \frac{1}{4 \sqrt{N}}$, we get for $N \rightarrow \infty $
$$
(1) \,\,S_r(q) - c_r \left(-2 \pi i \tau\right)^{-r} - d_r \left(-2 \pi i \tau \right)^{-r +1} \ll N^{\frac{r}{2}-1},
$$
and
$$
(2) \,\, \widetilde{S}_r(q) - \frac{c_r}{2} \left(-2 \pi i \tau\right)^{-r} - \tilde{d}_r\left(-2 \pi i \tau\right)^{-r+1}  \ll N^{\frac{r}{2}-1},
$$
where the constants are given by:
\begin{align*}
c_r &= \zeta(r)\left(1 -2^{1-r}\right), \\
d_r &=  -\left( \frac{\zeta\left(r-2\right)\left(1-2^{3-r}\right)}{2} +\rho_{C}\zeta(r-1)\left(1 - 2^{1-r}\right)\right) ,\\
d_r^{\prime} &= -\left(\zeta\left(r-2\right)\left(1-2^{3-r}\right) +\frac{\rho_{R}}{2}\zeta(r-1)\left(1 - 2^{1-r}\right)\right) .
\end{align*}
\end{proposition}
\begin{proof}
From \cite{BM} we know that $S_r(q)$ obeys the above expression. For $\widetilde{S}_r(q)$, we proceed as follows: We perform a Mittag-Leffler expansion on the two functions
\benn
\frac{q^{\frac{rn}{2}}}{\left(1-q^n\right)^r} = \frac{1}{\left(-2 \pi i n \tau\right)^r} + \sum_{0 < j <r \atop{ j \equiv r \pmod{2}}}  \frac{\alpha_j}{\left(-2 \pi i n \tau\right)^j} + \sum_{0 < j \leq r \atop{ j \equiv r \pmod{2}}}  \frac{\alpha_j}{\left(-2 \pi i \right)^j}\sum_{k \geq 1} \frac{1}{n \tau - k} +\frac{1}{n \tau + k}
\eenn
for $\alpha_j$ some constants and
\benn
\frac{1}{1 + q^n} = \frac{1}{2} +  \frac{-1}{2 \pi i}\sum_{k \geq 1} \frac{1}{n\tau -\left(k - \frac{1}{2}\right)} +\frac{1}{n\tau -\left(k + \frac{1}{2}\right)}.
\eenn
We insert the Mittag-Leffler expansion and split the sum in the following way
\benn
\widetilde{S}_r(q) = \Sigma_1 + \Sigma_2,
\eenn
where 
\benn
\Sigma_1 = \frac{1}{2} \sum_{n \geq 1} \frac{(-1)^{n+1} q^{n^2 + (\frac{r}{2}+\rho_{R})n}}{\left(1-q^n\right)^r},
\eenn
\benn
\Sigma_2 = S_1 + S_2 + S_3, 
\eenn
Here
\begin{align}
S_1 & := \sum_{n \geq 1 } (-1)^{n+1} q^{n^2 + \rho_R n} \frac{1}{(-2\pi i n \tau)^r} \frac{1}{-2 \pi i} \sum_{k \geq 1} \frac{1}{n \tau - \left(k-\frac{1}{2}\right)}+ \frac{1}{n \tau + \left(k-\frac{1}{2}\right)}, \nonumber \\
S_2 & := \sum_{n \geq 1 } (-1)^{n+1} q^{n^2 + \rho_R n} \sum_{0 < j < r \atop{ j \equiv r \pmod{2}}} \frac{\alpha_j}{\left(-2 \pi i n \tau\right)^j} \frac{1}{(-2\pi i)} \sum_{k \geq 1} \frac{1}{n \tau - \left(k-\frac{1}{2}\right)}+ \frac{1}{n \tau + \left(k-\frac{1}{2}\right)},\nonumber \\
\textrm{and} \nonumber \\ 
S_3 & := \sum_{n \geq 1 } (-1)^{n+1} q^{n^2 + \rho_R n} \sum_{0 < j \leq r \atop{ j \equiv r \pmod{2}}} \frac{\alpha_j}{\left(-2 \pi i \right)^j}\sum_{k \geq 1}\frac{1}{n \tau - k}+\frac{1}{n \tau + k} \nonumber\\
 \qquad &  \qquad \times \frac{1}{(-2\pi i)} \sum_{k \geq 1} \frac{1}{n \tau - \left(k-\frac{1}{2}\right)}+ \frac{1}{n \tau + \left(k-\frac{1}{2}\right)}. \nonumber
\end{align}
To proceed, we further define:
\be
g_k(\tau):= \sum_{n \geq 1} (-1)^{n+1} n^{-k} q^{n^2 + \rho_R n}.
\ee
With this, we can rewrite
\begin{align} \label{equ1}
& \Sigma_1 = \frac{1}{2}\frac{1}{\left(-2 \pi i \tau\right)^r}g_r(\tau) +  \frac{1}{2}\sum_{0 <j <r \atop { r \equiv j \pmod{2}}}\frac{\alpha_j}{\left(-2 \pi i \tau\right)^j}g_j(\tau) \nonumber \\ 
\qquad & \qquad + \sum_{n \geq 1} (-1)^{n+1}q^{n^2 + \rho_R} \sum_{ 0 < j \leq r} \frac{\alpha_j}{(-2 \pi i)^j}\sum_{k \geq 1} \frac{1}{(n \tau - k)^j} + \frac{1}{(n \tau + k)^j}.
\end{align}
If $j \geq 1$, then $g_j$ is convergent at $\tau = 0$, with the following value:
\benn
g_{j}(0)= \sum_{n \geq 1} \frac{(-1)^{n+1}}{n^j} = \zeta(j)\left(1 - 2^{1-j} \right).
\eenn
If $j=1$ we have the alternating sum, which is not absolutely convergent, but for $j \geq 2$ $g_j$ is absolutely and uniformly convergent for all $|q| \leq 1$, since we can bound the absolute value of the function by the Riemann zeta function. Next, we use a Taylor expansion to 
obtain lower order asymptotic terms. First, we compute the derivatives of $g_j$
\begin{align*}
 \frac{1}{2 \pi i}\frac{\partial}{\partial \tau} g_j(\tau)& = g_{j-2}(\tau) + \rho_R g_ {j-1}(\tau), \\
 \left(\frac{1}{2 \pi i}\right)^2 \frac{\partial^2}{\partial \tau^2} g_j(\tau) & = g_{j-4} + 2 \rho_R g_{j-3} + \rho_R^2 g_{j-2}. 
\end{align*}
So it is possible to use Taylor's theorem directly for $ j \geq 6 $, giving 
\benn
g_{j}(\tau)-g_{j}(0) - g^{\prime}(0)\tau \ll |\tau|^2 \sup_{\omega \in \H}\left| g_{j}^{\prime \prime}(\omega) \right| \ll |\tau|^2 \ll N^{-1}.
\eenn
We will replace the function $g_j$ by the constant term and the first term in the Taylor expansion by keeping track of the introduced error.
In (\ref{equ1}), we substitute in for the first term the Taylor expansion of $g_j$ and we bound the two other terms in the following way: For the second term in (\ref{equ1}) we bound 
$$
g_j(\tau) \ll \sum_{ n \geq 1} \frac{e^{-Cn^2 y}}{n^j} \ll N^{\frac{1}{4}}
$$
for some constant $C$. For the last term in (\ref{equ1}) we observe that for $ k \in \Z$ with the constraint on $\tau$ made in the proposition we have
\be \label{abs1}
\frac{1}{n \tau + k } \ll \frac{1}{k}.
\ee
Using this we obtain for $j>1$
$$
\sum_{n \geq 1}(-1)^{n+1}q^{n^2 + \rho_R n }\sum_{k \geq 1}\left(\frac{1}{n \tau + k}+\frac{1}{n \tau - k}\right)\ll \sum_{k \geq 1}k^{-j} \sum_{n \geq 1} e^{-2\pi n^2 y} \ll N^{\frac{1}{4}},
$$
by comparing the second sum with an Gaussian integral. \\ 
For $j=1$ we have to use the other bound, with $r \in \Z + \frac{1}{2}\Z$,
\be \label{abs2}
\frac{1}{n \tau + r} + \frac{1}{n \tau - r}  \ll \frac{n |\tau|}{r^2} 
\ee
to obtain
$$
\sum_{n \geq 1}(-1)^{n+1}q^{n^2 + \rho_R n }\sum_{k \geq 1}\left(\frac{1}{n \tau + k}+\frac{1}{n \tau - k}\right)  \ll |\tau| \sum_{n \geq 1} n e^{-2\pi n^2 y} \ll 1.
$$
Inserting the Taylor expansion we obtain: 
$$
\Sigma_1 = \frac{1}{2} \frac{\zeta(r)(1-2^{1-r})}{(-2\pi i \tau)^r} - \frac{1}{2}\left(\frac{\zeta(r-1)(1-2^{2-r}) + \rho_R \zeta(r-2)(1 -2^{3-r})}{(-2\pi i \tau)^{r-1}}\right) + O\left(N^{\frac{r}{2}-1}\right).
$$
For $\Sigma_2$ we have to give bounds for $S_1, S_2$ and $S_3$. We start with $S_3$. Using the bounds (\ref{abs1}) and (\ref{abs2})
%\begin{align*}
%S_3 = \sum_{0 < j \leq r \atop{j \equiv r \pmod{2}}}\sum_{n\geq1}(-1)^{n+1}q^{n^2 + \rho_R n}& \sum_{k\geq1}\frac{1}{(n \tau -k)^j}+\frac{1}{(n \tau +k)^j} \\
% & \times \frac{1}{-2 \pi i}\sum_{k\geq1}\frac{1}{n \tau -(k-\frac{1}{2})}+\frac{1}{n \tau +(k-\frac{1}{2})}
%\end{align*}  
we obtain for $j>1$:
$$
S_3 \ll |\tau| \sum_{n \geq 1} n e^{-2 \pi n^2 y} \ll 1
$$
and for $j=1$: 
$$
S_3 \ll |\tau|^2 \sum_{n\geq 1} n^2 e^{-\pi n^2 y}\ll |\tau|^2\sum_{n \geq 1}n^3 e^{-\pi n^2 y} = O(1).
$$
For $S_2$, we use again the trivial bound for $g_j$, except for $j=1$, where we compare the occurring function with a Gaussian integral. We obtain at worst, the following asymptotic term:
$$
S_2 = O\left(N^{\frac{r}{2}-1}\right).
$$
For $S_1$ we calculate
\begin{align*}
 S_1 & = \sum_{n\geq 1} (-1)^{n+1} q^{n^2 + \rho_R n} \frac{1}{(-2\pi i \tau)^r} \frac{1}{(-2\pi i)}\sum_{k \geq 1} \left(\frac{1}{n \tau - \left(k - \frac{1}{2}\right)}+ \frac{1}{n \tau - \left(k + \frac{1}{2}\right)}\right) \\
     & = 2 \sum_{n\geq 1} (-1)^{n+1} q^{n^2 + \rho_R n} \frac{1}{(-2\pi i \tau)^r} \frac{1}{(-2\pi i)}\sum_{k \geq 1}\frac{n \tau}{n^2 \tau^2 - \left(k-\frac{1}{2}\right)^2}\\
     & = - \frac{2}{(-2\pi i \tau)^{r-1} 4 \pi^2} g_{r-1}(\tau)\sum_{k \geq 1} \frac{1}{n^2 \tau^2-\left(k- \frac{1}{2}\right)^2}+ \frac{1}{\left(k- \frac{1}{2}\right)^2}-\frac{1}{\left(k- \frac{1}{2}\right)^2} \\
     & \leq  - \frac{1}{(-2\pi i \tau)^{r-1} 2 \pi^2} g_{r-1}(\tau) \left( -\frac{\pi^2}{2} + \pi^2\right) \\
     & = - \frac{\frac{1}{2}}{(-2\pi i \tau)^{r-1}}\zeta(r-1)\left(1-2^{2-r}\right) + O\left(N^{\frac{r}{2}-1}\right),
\end{align*}
where we used again a Taylor expansion to replace $g_{r-1}$ by a constant up to an error of lower asymptotic order. For the cases $3 \leq r \leq 5$ we use the same argument as in the appendix of \cite{BM}.
\end{proof}
\begin{proposition}\label{mt}
Let $\mathcal{C}_r\,\left( \mathcal{R}_r\right)$ be the $r$-th positive symmetrized crank (rank) moment. Letting  
$|x|\leq y := \frac{1}{4 \sqrt{N}}$, we get for $N \rightarrow \infty $
$$
\mathcal{SC}_r(q) = c_r \sqrt{\frac{-i \tau}{2}}e^{\frac{\pi i }{8 \tau}}\left(-2 \pi i \tau\right)^{-r} + d_r \sqrt{\frac{-i \tau}{2}}e^{\frac{\pi i }{8 \tau}} \left(-2 \pi i \tau \right)^{-r +1} + O\left(N^{\frac{r}{2}-\frac{5}{4}}e^{\frac{\pi \sqrt{N}}{2}}\right)
$$
and
$$
\mathcal{SR}_r(q) = c_r \sqrt{\frac{-i \tau}{2}}e^{\frac{\pi i }{8 \tau}}\left(-2 \pi i \tau\right)^{-r} + \tilde{d}_r \sqrt{\frac{-i \tau}{2}}e^{\frac{\pi i }{8 \tau}}\left(-2 \pi i \tau\right)^{-r+1}  + O\left(N^{\frac{r}{2}-\frac{5}{4}}e^{\frac{\pi \sqrt{N}}{2}}\right)
$$
\end{proposition}
\begin{remark}
 We have chosen the variables in such a way that the limit  $N\rightarrow \infty$ corresponds to the essential singularity $q \rightarrow 1$.
\end{remark}
\begin{proof}[Proof of Proposition \ref{mt}]
To begin, we must account for the automorphic factor that controls the exponential singularities. Using the Dedekind inversion formula, we can deduce that:
$$
\frac{(-q)_{\infty}}{(q)_{\infty}} = \sqrt{\frac{- i \tau}{2}}\left(1 + O\left(e^{-2\pi \sqrt{N}}\right)\right).
$$
Making all substitutions we obtain: 
\begin{align*}
 \mathcal{SC}_r (q) & = \frac{(-q)_{\infty}}{(q)_{\infty}} S_r(q) = \sqrt{\frac{- i \tau}{2}}\left(1 + O\left(e^{-2\pi \sqrt{N}}\right)\right) \times \\ & \qquad \qquad \qquad \qquad \left( c_r\left(-2\pi i \tau\right)^{-r} + d_r\left(- 2\pi i\tau\right)^{-r +1} + O\left(N^{\frac{r}{2}-1}\right)\right) \\
& =  c_r \sqrt{\frac{-i \tau}{2}}e^{\frac{\pi i }{8 \tau}}\left(-2 \pi i \tau\right)^{-r} + d_r \sqrt{\frac{-i \tau}{2}}e^{\frac{\pi i }{8 \tau}} \left(-2 \pi i \tau \right)^{-r +1} + O\left(N^{\frac{r}{2}-\frac{5}{4}}e^{\frac{\pi \sqrt{N}}{2}}\right).
\end{align*}
Doing the same for the rank we obtain the desired result.
\end{proof}
\subsection{Bounds away from the dominant pole}
The next step is to make sure that the functions away from the pole $q=1$ can be estimated and bounded. Therefore, we have the following proposition:
\begin{proposition}
 For $y= \frac{1}{4 \sqrt{N}}$ and $y \leq |x| \leq \frac{1}{2}$ we get
\begin{align*}
\mathcal{SC}_r(q) &\ll N^{\frac{r}{2}+ \frac{1}{4}}e^{\frac{\pi \sqrt{N}}{4}}, \\
\mathcal{SR}_r(q) &\ll N^{\frac{r}{2}+ \frac{1}{4}}e^{\frac{\pi \sqrt{N}}{4}}.
\end{align*}
\end{proposition}
\begin{proof}
 In \cite{BM} the authors showed that $S_r(q) \ll N^{\frac{r}{2}+\frac{1}{4}}$. From that we obtain:
$$
\left|\mathcal{SC}_r(q) \right|\leq \left|\frac{(-q)_{\infty}}{(q)_{\infty}}\right| \left|S_r(q)\right| \leq N^{\frac{r}{2}+\frac{1}{4}}e^{\frac{\pi \sqrt{N}}{4}}.
$$
The same argument of \cite{BM} can be used to show that $\widetilde{S}_r(q)$ obeys the same bound as $S_r$ away from the dominant pole and therefore $\mathcal{SR}_r(q)$ obeys the same bound as $\mathcal{SC}_r(q)$. We see that this error term can be absorbed into the error of the asymptotics near the essential singularity $q=1$.
\end{proof}
\section{Circle Method}
With these bounds, it is now possible to obtain asymptotic formulas of the symmetrized moments using the Circle Method. We obtain the following result:
\begin{theorem}\label{th2}
For $N \rightarrow \infty$, we have:
\begin{align}
\overline{\mu}_r^+(N) & = \tilde{c}_r N^{\frac{r}{2}-\frac{3}{4}}I_{r-\frac{3}{2}}\left(\pi \sqrt{N}\right) + \tilde{d}_r N^{\frac{r}{2}-\frac{5}{4}}I_{r-\frac{5}{2}}\left(\pi \sqrt{N}\right) + O\left(N^{\frac{r}{2}-\frac{7}{4}e^{\pi \sqrt{N}}}\right) ,\\
\overline{\eta}_r^+(N)& = \tilde{c}_r N^{\frac{r}{2}-\frac{3}{4}}I_{r-\frac{3}{2}}\left(\pi \sqrt{N}\right) + \tilde{d}_r^{\prime} N^{\frac{r}{2}-\frac{5}{4}}I_{r-\frac{5}{2}}\left(\pi \sqrt{N}\right) + O\left(N^{\frac{r}{2}-\frac{7}{4}e^{\pi \sqrt{N}}}\right),
\end{align}
where $I_s$ is the order $s$ modified Bessel function of the first kind.
\end{theorem}
\begin{proof}
To compute the Fourier coefficients, we use Cauchy's theorem with the circle parametrized as $q = e^{-\frac{ \pi}{2 \sqrt{N}} + 2\pi i x}$:
$$
\overline{\mu}_r^{+} = \frac{1}{2 \pi i} \int_{\mathcal{C}} \frac{\mathcal{C}_r(q)}{q^{N+1}}dq = \int_{-\frac{1}{2}}^{\frac{1}{2}} \mathcal{C}_r \left(e^{-\frac{ \pi}{2 \sqrt{n}} + 2\pi i x}\right) e^{- 2\pi i N x + \frac{\pi \sqrt{N}}{2}}dx
$$
We split the integral in the following way:
$$
\int_{-\frac{1}{2}}^{\frac{1}{2}} = \int_{|x| \leq \frac{1}{4 \sqrt{N}}} + \int_{\frac{1}{4\sqrt{N}}\leq |x| \leq \frac{1}{2}}
$$ 
and show that the second sums can be absorbed into the error of the first sum, where $\tilde{I}_1$ is the integral referring to the first summand and analogous $\tilde{I}_2$ refers to the second summand of the splitting. We define $u:= 4\sqrt{N}x$ and obtain:
\begin{align*}
 \tilde{I}_1 &= \frac{1}{4 \sqrt{N}} \int_{-1}^{1} \mathcal{C}_r\left( e^{\frac{\pi i}{2 \sqrt{N}}}\right)e^{\frac{\pi \sqrt{N}}{2}(1 -iu)}du \\
     &= \frac{1}{4 \sqrt{N}} \int_{-1}^{1} \frac{c_r}{2 \sqrt{\pi}}\left(\frac{\pi }{2 \sqrt{N}}(1 -i u)\right)^{-r + \frac{1}{2}}e^{\frac{\pi  \sqrt{N}}{2}\left(\frac{1}{(i +u)} + 1 -iu\right)}du \\   
     \qquad&\qquad+ \frac{1}{4 \sqrt{N}} \int_{-1}^{1} \frac{d_r}{2 \sqrt{\pi}}\left(\frac{\pi }{2 \sqrt{N}}(i -i u)\right)^{-r + \frac{3}{2}}e^{\frac{\pi  \sqrt{N}}{2}\left(\frac{1}{(i +u)} + 1 -iu\right)}du \\
     \qquad&\qquad+ O\left(N^{\frac{r}{2} -\frac{5}{4}}e^{\pi \sqrt{N}}\right).	
\end{align*}
Choosing now $v:= 1-iu$ as a new variable, we obtain: 
$$\tilde{I}_1 = c_r \pi^{-r+1}2^{r-\frac{5}{2}}N^{\frac{r}{2}-\frac{3}{4}}P_{-r + \frac{1}{2}} + d_r \pi^{-r+2} 2^{r +\frac{7}{2}} N^{\frac{r}{2}-\frac{5}{4}}P_{-r + \frac{3}{2}} + O\left(N^{\frac{r}{2}-\frac{7}{4}}e^{\pi \sqrt{N}}\right),
$$
where we define
$$
P_s = \frac{1}{2 \pi i}\int_{1-i}^{1+i} v^s e^{\frac{\pi \sqrt{N}}{2}\left(v + \frac{1}{v}\right)}.
$$ 
In \cite{KKS}, it is proven that for large $N$ 
$$
P_{s} = I_{-s-1}(\pi \sqrt{N}) + O\left(e^{\frac{3\pi}{4}\sqrt{N}}\right),
$$
and that completes the proof for the positive symmetrized moments. The proof for the rank is analogous.
\end{proof}
Next we prove that away from the pole the error is small and can so be absorbed into the error of $\tilde{I}_1$.
\begin{proposition}
 For $N \rightarrow \infty$, we get 
$$
\tilde{I}_2 \ll N^{\frac{r}{2} +\frac{1}{4}}e^{\frac{3\pi}{4}\sqrt{N}}.
$$
\end{proposition}
\begin{proof}
 Using the bounds obtained away from the pole, we get the following result:
$$|\tilde{I}_2|\leq \int_{\frac{1}{4 \sqrt{N}}\leq |x| \leq \frac{1}{2}}\left| \mathcal{C}_{r}\left(e^{-\frac{\pi}{2 \sqrt{N}} + 2\pi i x}\right) e^{- 2\pi i N x + \frac{\pi \sqrt{N}}{2}}\right|dx \ll N^{\frac{r}{2} +\frac{1}{4}}e^{\frac{3\pi \sqrt{N}}{4}}.
$$
As the rank has the same bound away from the dominant pole the same result holds for the rank. 
\end{proof}
\begin{proof}[Proof of Theorem \ref{th1}]
We use the asymptotic formula for general index $s$ modified Bessel function of the first kind as $x \rightarrow \infty$
$$
I_s(x) = \frac{e^{x}}{\sqrt{2 \pi x}} + O\left(\frac{e^x}{x^{\frac{3}{2}}} \right).
$$
That settles directly part one of the theorem.
For part two we rewrite (using (\ref{smum}))
$$
\overline{M}^+_r(N)-\overline{N}^+_r(N) = r! \left(\overline{\mu}^+_r - \overline{\eta}^+_r \right) + \sum_{l=1} a_{\ell}\left( \overline{\mu}^+_{\ell} - \overline{\eta}^+_{\ell}\right)
$$
Using Theorem \ref{th2} every term in the sum is at most $O\left(N^{\frac{r}{2}-2}e^{\pi \sqrt{N}}\right)$. Inserting the asymptotics of the Bessel function yields the theorem for $r\geq3$. The result for $r=1$ is given in \cite{KKS} and the result for $r=2$ can be done with similar methods as in \cite{B}. Due to the length we skip the proof. 
\end{proof}
%For the next statement we define in terms of the crank and rank moments a new function that has an analogous counterpart in the theory of partitions.
%\begin{definition}
% Let $r >1 $. We define the $r-1$-th $spt$ function for overpartitions as 
%$$
%\overline{spt}_{r-1}(N):= 2 \cdot \left(\overline{M}^+_r(N) - \overline{N}^+_r(N) \right). 
%$$
%\end{definition}
\begin{proof}[Proof of corollary \ref{co1}]
 For $r=2$ the proof follows from \cite{BLO}. For $r\geq 3$ the proof follows from the fact that the constant in front of Theorem \ref{th1} (ii) is positive. For $r=1$ the inequality $\overline{M}^+_r(N) - \overline{N}^+_r(N)>0$ follows from
\cite{ABKO}.
\end{proof}

\end{document}